\documentclass[12pt,a4paper]{article}

\usepackage[T1]{fontenc}
\usepackage[latin1]{inputenc}

\usepackage{amsmath,amsthm,amsfonts,amssymb}
\usepackage{enumerate}
\usepackage{cleveref}

\usepackage{filecontents}
\usepackage[backend=bibtex,
style=numeric,
giveninits=true,
doi=false,
sorting=nyt,
maxnames=5]{biblatex}
\newcommand{\EQ}{\begin{equation}}
\newcommand{\EN}{\end{equation}}
\newtheoremstyle{examplestyle}
{\topsep}   
{\topsep}   
{\slshape}  
{0pt}       
{\bfseries} 
{.}         
{5pt plus 1pt minus 1pt} 
{}          
\newtheorem{theorem}{Theorem}[section]
\newtheorem{definition}[theorem]{Definition}

\newtheorem{proposition}[theorem]{Proposition}
\newtheorem{lemma}[theorem]{Lemma}

\newtheorem{remark}[theorem]{Remark}

\theoremstyle{examplestyle}

\newcommand{\F}{\mathbb{F}}

\renewcommand{\u}{{\mathbf{1}}}

\newcommand{\cS}{{\cal S}}

\newcommand{\bv}{\mathbf{v}}
\newcommand{\ba}{\mathbf{a}}
\newcommand{\bb}{\mathbf{b}}

\newcommand{\bc}{\mathbf{c}}

\newcommand{\bz}{\mathbf{z}}
\newcommand{\bu}{\mathbf{u}}
\newcommand{\bw}{\mathbf{w}}
\newcommand{\be}{\mathbf{e}}
\newcommand{\bx}{\mathbf{x}}
\newcommand{\by}{\mathbf{y}}

\newcommand{\bkappa}{\boldsymbol{\kappa}}

\newcommand{\wt}{\operatorname{wt}}
\newcommand{\Aut}{\operatorname{Aut}}

\newcommand{\Supp}{\operatorname{supp}}
\newcommand{\Rank}{\operatorname{rank}}
\newcommand{\Gcd}{\operatorname{gcd}}
\newcommand{\Degree}{\operatorname{degree}}

\newcommand{\HFP}{\operatorname{HFP}}

\title{Circulant Hadamard matrices as $\HFP$-codes of type $C_{4n}\times C_2$.\footnote{This work has been partially supported by the Spanish MICINN grants TIN2016-77918-P, MTM2015-69138-REDT and the Catalan AGAUR grant 2014SGR-691.}}
\author{J. Rif\`{a} \\
	Department of Information and Communications Engineering,\\
	Universitat Aut\`{o}noma de Barcelona}

\begin{document}
	\maketitle

\begin{abstract}
	We prove that a circulant Hadamard code of length $4n$ can always be seen as an $\HFP$-code (Hadamard full propelinear code) of type $C_{4n}\times C_2$, where $C_2=\langle \bu\rangle$ or the same, as a cocyclic Hadamard code. We compute the rank and dimension of the kernel of these kind of codes.
\end{abstract}

\section{Introduction}

Let $\F$ be the binary field.
For any $\bv\in \F$, we define the support of $\bv$ as the set of nonzero positions of $\bv$ and we denote it by $\Supp(\bv)$. For a vector
$\bx \in \F^n$ denote by $\wt(\bx)$ its \textit{Hamming weight} (i.e. the number
of its nonzero positions).
For two vectors $\bx=(x_1,\ldots, x_n)$ and $\by = (y_1, \ldots, y_n)$
from $\F^n$ denote by $d(\bx, \by)$ the \textit{Hamming distance} between $\bx$
and $\by$ (i.e. the number of positions $i$, where $x_i \neq y_i$).

A binary code $C$ of length $n$ is a subset of $\F^n$. Two structural properties of nonlinear codes
are the rank and dimension of the kernel. The rank of a binary code $C$, $r =
\Rank(C)$, is the dimension of the linear span of $C$.  The kernel of a binary code $C$ is defined
as $K(C) = \{x \in \F^n : x + C = C\}$. 
Let $\be, \bu$ be the binary all zeros vector and all ones vector, respectively. If $\be\in C$ then $K(C)$ is a linear subspace of $C$ with dimension $k$. We use
$\cS_n$ to denote the group of all coordinate permutations.

\begin{definition}
	An Hadamard matrix is an $4n\times 4n$ matrix $H$ containing entries from the
	set $\{1,-1\}$, with the property that:
	$$
	HH^T = 4nI,
	$$
	where $I$ is the identity matrix.
\end{definition}

Two Hadamard matrices are {\it equivalent} if one can be obtained
from the other by permuting rows and/or columns and multiplying them
by $-1$. We can change the first row and
column of $H$ into $+1$'s and we obtain an equivalent Hadamard
matrix which  is called {\it normalized}.
Form a normalized Hadamard matrix, if $+1$'s are replaced by 0's and $-1$'s by
1's, we obtain a {\it (binary) Hadamard matrix
	$H'$} with orthogonal rows and where any two rows agree in $2n$ places and differ in $2n$ places. The binary code
consisting of the rows of $H'$ and their complements is called
a {\it (binary) Hadamard code} \cite{McWill}.

Next lemmas are well known. We include them without proof.

\begin{lemma}\cite{prv1,prv3}\label{odd}
	Let $C$ be an Hadamard code of length $2^sn'$, where $n'$ is odd. The dimension of the kernel is $1\leq k\leq s-1$.
\end{lemma}
\begin{lemma}\cite[Th. 2.4.1 and Th. 7.4.1]{ak}\label{aske}
	Let $C$ be an Hadamard code of length $4n=2^sn'$, where $n'$ is odd.
	\begin{enumerate}[(i)]
		\item If $s\geq 3$ then the rank of $C$ is $r\leq 2n$, with equality if $s=3$.
		\item If $s=2$ then $r\geq 4n-1$.
	\end{enumerate}
\end{lemma}
In an Hadamard code we always have $\{\be,\bu\}\subset K(C)$. We will say that the kernel is trivial when $\{\be,\bu\}= K(C)$.
\begin{lemma}\cite{RS17}\label{proj}
	Let $C$ be a non linear Hadamard code of length $4n$ with non trivial kernel. Let $\bkappa\in K(C)$ be such that $\bkappa\notin \{\be,\bu\}$. Then, the projection of $C$ onto $\Supp(\bkappa)$ is an Hadamard code of length $2n$.
\end{lemma}
Hadamard conjecture asserts that an Hadamard matrix of order $4n$ exists for every positive integer $n$. 
In this paper we are mainly interested in circulant
Hadamard matrices.

\begin{definition}\label{def:circulant}
	A circulant Hadamard matrix of order $4n$ is a square matrix of the form
	\begin{equation}\label{eq:circ}
	H=\begin{pmatrix}
	a_1&a_2&\cdots & a_{4n}\\
	a_{4n}&a_1&\cdots &a_{4n-1}\\
	\cdots &\cdots &\cdots& \cdots \\
	a_2&a_3&\cdots &a_1
	\end{pmatrix}
	\end{equation}
	with $a_i \in \{-1, 1\}$ for all index $i\in \{1,2,\ldots,4n\}$, and $H H^T = 4nI$. 
	We call $g$ the binary generator polynomial of the circulant matrix $H$, where $g=g_1+g_2x+\ldots+g_{4n}x^{4n-1}\in \F[x]/x^{4n}-1$ and $g_i=0$ when $a_i=1$, $g_i=1$ when $a_i=-1$.
\end{definition}

No circulant Hadamard matrix of order larger
than 4 has ever been found, but the nonexistence is still a non proved result. It seems that the first time this result was stated as a conjecture was in \cite{rys}. The most important result about circulant Hadamard matrices is the following one, where the proof uses the algebraic number theory and no equivalent proof using only elementary commutative algebra is known (except for specific values of $n$, as for example in \cite{S2013}).

\begin{proposition}\cite{tur}\label{prop:circulant}
	Let $H$ be a circulant Hadamard matrix of order $4n$. Then $n$ is an odd square number.
\end{proposition}
\begin{definition}\cite{rbh}
	A binary code $C$ of length $n$ has a  propelinear structure if for each codeword
	$\bx\in C$ there exists $\pi_\bx\in \cS_n$ satisfying the following conditions:
	\begin{enumerate}
		\item For all $\by\in C$, $\bx+\pi_\bx(\by)\in C$,
		\item For all $\by\in C$, $\pi_\bx\pi_\by=\pi_\bz$, where $\bz = \bx +\pi_\bx(\by)$.
	\end{enumerate}
\end{definition}

For all $\bx\in C$ and for all $\by\in \F^n$, denote by $*$ the binary operation
such that $\bx*\by=\bx+\pi_\bx(\by)$. Then, $(C,*)$ is a group, which is not
abelian in general. Vector $\be$ is always a codeword and $\pi_\be$ is the
identity permutation. Hence, $\be$ is the identity element in $C$ and $\bx^{-1}=\pi_\bx^{-1}(\bx)$, for all $\bx\in C$~\cite{rbh}. We call $C$ a propelinear code if it has a propelinear structure.

We call $C$ a propelinear Hadamard code if $C$ is both, a propelinear code and an Hadamard code. We say that the group structure of $C$ is the type of $C$.

\begin{definition}
	An Hadamard full propelinear code ($\HFP$-code) is an Hadamard propelinear code $C$ such that for every $\ba\in C$, $\ba\notin \{\be,\bu\}$ the permutation $\pi_\ba$ has not any fixed coordinate and $\pi_\be=\pi_\bu=I$.
\end{definition}

In \cite{RS17} $\HFP$-codes were introduced and their equivalence with Hadamard groups was proven; on the other hand,  the equivalence of Hadamard groups, relative $(4n,2,4n,2n)$-difference sets in a group, and cocyclic Hadamard matrices, is already known (see \cite{hor,RS17} an references there).

Next lemma is well known, we will use it later.
\begin{lemma}\cite{bmrs}\label{prop:3.2}
	Let $(C,*)$ be a propelinear code of length $4n$.
	\begin{enumerate}[i)]
		\item  For $\bx\in C$ we have $\bx \in K(C)$ if and only if $\pi_\bx \in \Aut(C)$.
		\item The kernel $K(C)$ is a subgroup of $C$ and also a binary linear space.
		\item If $\bc\in C$ then $\pi_\bc\in Aut(K(C))$ and $\bc * K(C)=\bc+K(C)$.
	\end{enumerate}
\end{lemma}

The unique (up to equivalence) circulant Hadamard matrix of order 4 can be presented as:
\begin{equation}\label{eq:circ1}
\begin{pmatrix}
1&0&0&0\\0&1&0&0\\0&0&1&0\\0&0&0&1
\end{pmatrix}
\end{equation}

or, after normalization, as:
\begin{equation}\label{eq:circ2}
\begin{pmatrix}
0&0&0&0\\1&1&0&0\\1&0&1&0\\1&0&0&1
\end{pmatrix}.
\end{equation}
Now, taking $\ba=(1100)$ with associated permutation $\pi_{\ba}=(1,2,3,4)$ we see that $C=\langle \ba, \bu \rangle$ is an $\HFP$-code of type $C_4 \times C_{2\bu}$, where $C_{2\bu}$ means a cyclic group of order two containing $\bu$. It is also easy to see that taking $\ba=(1100)$  and  $\bb=(1010)$, with associated permutations $\pi_{\ba}=(1,2)(3,4)$ and $\pi_{\bb}=(1,3)(2,4)$, respectively, we obtain an $\HFP$-code $C=\langle \ba, \bb, \bu \rangle$ of type $C_2 \times C_2 \times C_{2\bu}$.

In this paper we prove that a circulant Hadamard code can always be seen as an $\HFP$-code of type $C_{4n}\times C_{2\bu}$ and vice versa, an $\HFP$-code of the above type always gives a circulant Hadamard code. 

\section{Circulant Hadamard full propelinear codes}\label{sec2}

A circulant $4n\times 4n$ Hadamard matrix (\Cref{eq:circ}) is a matrix $H$ with rows given, in polynomial way, by $g,xg,\ldots, x^{4n-1}g\in \F[x]/x^{4n}-1$, where $g\in \F[x]/x^{4n}-1$.  A circulant Hadamard code $C$ (with $\be\in C$) is given by the vectors corresponding to the coefficients of polynomials in $\{g+x^ig, \bu+g+x^ig: 1\leq i \leq 4n\}$, where $\bu$ is, in a polynomial way, $\bu=1+x+x^2+\ldots +x^{4n-1}$.

Next lemma showing that $n$ is a square is well known, but we include the proof here since we need the second result of it.
\begin{lemma}
	\label{lemm:square}
	Let $H$ be a circulant  $4n\times 4n$ Hadamard matrix. Then, $n$ is a square and the weight of each column in the corresponding binary matrix is $2n\pm \sqrt{n}$.
\end{lemma}
\begin{proof}
	Let $g\in  \F[x]/x^{4n}-1$ be the binary generator polynomial of the circulant Hadamard matrix $H$.	Assume that polynomial $g$ has weight $2n+ \sigma$ (note that we are talking about the weight of the vector given by the coefficients of the polynomial). Let $J$ be the matrix with all entries equal to one. Compute $H^TJ$ which gives $(-2\sigma)J$. Also $HJ=(-2\sigma)J$. Hence $4nIJ=HH^TJ=(-2\sigma)^2J$, so $4n=4\sigma^2$ and $n$ is a square. Further, $\sigma=\pm \sqrt{n}$ and so, the weight of each column in $H$, considered as a binary matrix, is $2n\pm \sqrt{n}$.
\end{proof}

Now, our interest is to show that a circulant Hadamard code $C$ always has a full propelinear structure. 

\begin{proposition}\label{prop:str}
	Let $H$ be a circulant Hadamard code of order $4n$, with binary generator polynomial $g$. Code $C=\{g+x^ig+\xi\u \,|\, 0\leq i \leq n-1; \xi \in \F\}$, where $\u$ is the polynomial with all coefficients one,  equipped with the operation $*$ defined by
	\begin{equation}\label{eq:str}
	\begin{split}
	(g+x^ig+\xi\u)*(g+x^jg+\xi'\u)=g+x^{i+j}g+(\xi+ \xi')\u \\
	\mbox{for any $i,j\in \{0,1,\ldots,4n-1\}$ and  $\xi,\xi'\in \F$},
	\end{split}
	\end{equation}
	is an $\HFP$-code with a cyclic group of permutations $\Pi$ of order $4n$.
\end{proposition}

\begin{proof}
	First, for all $i\in \{0\ldots 4n-1\}$ take the permutation associated to the element $z=g+x^ig+\xi\u$ as $\pi_z=x^i$, where we use $x^i$ meaning a cyclic shift of $i$ positions. Indeed, with these permutations we obtain
	$(g+x^ig+\xi\u)*(g+x^jg+\xi^\prime \u)= g+x^ig+\xi\u+x^i(g+x^jg+\xi^\prime\u) = g+ x^{i+j}g +(\xi+\xi^\prime)\u$. Now, it is clear that for all $a,b\in C$ we have $a*b\in C$.
	
	On the other hand,
	for all $a=g+x^ig+\xi\u$,
	$b=g+x^jg+\xi^\prime\u\in C$
	we have $a*b=g+x^{i+j}g+(\xi+\xi^\prime)\u$ and so, $\pi_{a*b}=\pi_a\pi_b$.
	Hence, code $C$ is propelinear and, further, is full propelinear since all the associated permutations $\pi_a$ have no any fixed point, with the exception of $a\in\{\be,\bu\}$.
	
	Finally, we note that $\Pi$ is a cyclic group of order $4n$ and the propelinear structure of $C$ is given by $C_{4n}\times C_2$, where $\bu\in C_2$. The elements in the propelinear code are $\ba, \ba^2, \ldots, \ba^{4n-1}$ and the complements where $\ba^i$, in polynomial way, is $(1+x^i)g$.
	
	The constructed propelinear code is Hadamard. Indeed, we know that the initial circulant matrix $H$ is Hadamard and so the weight of each element $\ba^i$ is $\wt(\ba^i)=d(g,x^ig)=2n$. Also, for $i\not= j$ we have $d(\ba^i, \ba^j)= \wt(\ba^i+\ba^j)=\wt(x^i(g+x^{j-i}g))=\wt(\ba^{j-i})=2n$ .
\end{proof}

Vice versa, next proposition shows that an $\HFP$-code of type $C_{4n} \times C_{2\bu}$, or the same, an $\HFP$-code with a cyclic group of permutations of order $4n$, gives a circulant Hadamard code.

\begin{proposition}\label{prop:str1}
	Let $C$ be an $\HFP$-code of length $4n$ with a cyclic group of permutations $\Pi$ of order $4n$, so an $\HFP$-code of type $C_{4n}\times C_{2\bu}$. Then, $C$ is Hadamard circulant code of length $4n$.
\end{proposition}
\begin{proof}
	Let $\ba\in C$ be a generator of $C_{4n}$. The elements in the propelinear code are $\ba, \ba^2, \ldots, \ba^{4n-1}$, and the complements. As the weight of $\ba$ is $2n$ the associated polynomial $a(x)$ is divisible by $(1+x)$. Take as $g$ the quotient polynomial after dividing $a(x)$ by $(1+x)$. Note that the element $\ba^i$ can be written as $\ba^i = \ba+\pi_{\ba}(\ba)+\pi_{\ba}^2(\ba)+\ldots + \pi_{\ba}^{i-1}(\ba)$ and, in polynomial way, $(1+x+x^2+\ldots +x^{i-1})a(x)=\frac{1+x^i}{1+x}a(x)=(1+x^i)g$. The matrix given by the rows $x^ig$ is a circulant Hadamard matrix.
\end{proof}
The unique circulant Hadamard code we know is given by the rows of matrix in \Cref{eq:circ2} and their complements, hence the code is linear and so the values of the rank and dimension of the kernel are $r=k=3$.

\begin{lemma}\label{lemm:2.4}
Let $C$ be an $\HFP$-code of type $C_{4n}\times C_{2\bu}$, where $C=\langle \ba,\bu \rangle$. Vector $\bu$ belongs to the linear span of $\langle \ba \rangle$.
\end{lemma}
\begin{proof}
Adding the $4n-1$ vectors $\ba,\ba^2,\ldots, \ba^{4n-1}$, we obtain in polynomial form (see the proof of \Cref{prop:str1}) $\sum_{i=1}^{4n-1} (1+x^i)g$ and since $4n-1$ is odd the above addition coincides with the addition of all rows in the binary circulant matrix. From \Cref{lemm:square} the number of ones in each column is either $2n+ \sqrt{n}$ or $2n- \sqrt{n}$, so we obtain the all ones vector $\bu$. This proves the statement.
\end{proof}

Next proposition characterizes the values of rank and dimension of the kernel for circulant Hadamard codes. The proof uses the strong result given in \Cref{prop:circulant}.

\begin{proposition}\label{prop:rank}
	Let $C$ be a non linear circulant Hadamard code of length $4n$. Then the dimension of the kernel is $k=1$ and the rank is $r=\Rank(C)=4n-1$.
\end{proposition}
\begin{proof}
	From \Cref{odd} if $C$ is an Hadamard code of length $4n=2^sn'$, with odd $n'$ then the dimension of the kernel is $1\leq k\leq s-1$. Hence in our case, from \Cref{prop:circulant}, as $n$ is odd and $s=2$ we have $k=1$.
	
	For the rank, from \Cref{prop:circulant} we have that $n$ is odd and from \Cref{aske}  we have  $r\geq 4n-1$. From \Cref{lemm:2.4} vector $\bu$ is a linear combination of $\ba,\ba^2,\ldots, \ba^{4n-1}$, so $r\leq 4n-1$. This proves the statement.
\end{proof} 

Without using the result in \Cref{prop:circulant} we can give an alternative proof about the dimension of the kernel as we see in the following \Cref{rem:1}. However, the computation of rank in \Cref{prop:rank} is absolutely tied to the condition about parity of $n$. 
Nevertheless, for the special case when $n$ ($n\not= 1$) is a power of two, we can obtain the nonexistence of circulant codes as we prove in \Cref{rem:2}.

\begin{remark}\label{rem:1}
	Let $C$ be a non linear circulant Hadamard code of length $4n=2^sn'$, with $n'$ odd. Then the dimension of the kernel is $k=1$.
\end{remark}
\begin{proof}
	Indeed, assume that the dimension of the kernel is $k\not=1$ and let $\bv$ be a vector in the kernel of $C$ different form $\be$ and $\bu$. Take $\bc=(101010\ldots10)$. For any binary vector $\bx$ of length $4n$ say that $\bx^{(1)}, \bx^{(2)}$ are the projections over the first and the second half part of $\bx$, respectively.
	
	First of all, we want to prove that $\ba^{2n}=\bc\in K(C)$.
	Let  $\bv$ a vector in the kernel of $C$ different from $\be$ and $\bu$. If $\bv^{(1)}=\bv^{(2)}$  take $\bkappa_1=\bv$; if $\bv^{(1)}=\bv^{(2)}+\bu$  take $\bkappa_1=\bv+\pi_{\ba^n}(\bv)$; otherwise take $\bkappa_1=\bv+\pi_{\ba^{2n}}(\bv)$, which from \Cref{prop:3.2} belongs to the kernel and $\bkappa_1^{(1)}=\bkappa_1^{(2)}$.
	In any case, $\bkappa_1\notin \{\be,\bu\}$ and so $\bkappa_1+\pi_{\ba}(\bkappa_1)\not=\be$. If $\bkappa_1+\pi_{\ba}(\bkappa_1)=\bu$ then $\bkappa_1=\bc$. 
	In this last case $\bkappa_1$ is of order a power of two and so $\bkappa_1=\ba^{2^tn'}$, for $2\leq t\leq s-1$. Computing $\bkappa_1^{2}$ we obtain $\be$ and so $\bkappa_1=\ba^{2n}$. If we are not in the above case we can use the same argumentation as before using $\bkappa_1^{(1)}$ instead of $\bv$. So, if $\bkappa_1^{{(1)}{(1)}}=\bkappa_1^{{(1)}{(2)}}$ then take $\bkappa_2=\bkappa_1$, otherwise take $\bkappa_2=\bkappa_1+\pi_{\ba^{n}}(\bkappa_1)$ which belongs to the kernel and $\bkappa_2^{{(1)}{(1)}}=\bkappa_2^{{(1)}{(2)}}$. We can repeat the same construction until we can get a vector in the kernel which can be divided in $2^{s-1}$ parts which are exactly equal. This means that the last constructed vector, say it is $\bkappa$, belongs to the kernel and in each $2^{s-1}$th part has weight $n'$. 
	If $n'\not=1$, take $\bw=\bkappa+\pi_\ba(\bkappa)$ which is still into the kernel. However, the weight of $\bw$ is $2^{s-1}$ multiplied by and even number which is impossible. Hence, we conclude that the length $4n$ of the code $C$ is a power of two and, moreover, $\bkappa=\bc\in K(C)$. 
	Therefore, in any case, $\bc\in K(C)$. Vector $\bc$ is of the form $\ba^i$ for some index $i$ and we should have $\ba^{2i}\in \{\be,\bu \}$. Due to the type of code $C$ the case $\ba^{2i}=\bu$ is not possible, hence $\ba^{2i}=\be$ and so $i=2n$ and $\bc=\ba^{2n}$.

	Finally, take the projection of $\ba^{2i}$, for $i\in \{0\ldots 2n-1\}$, over the support of $\ba^{2n}$.
	The vectors $\ba^{2i}_p$ we obtain form a circulant Hadamard code of order $2n$ which contradicts \Cref{lemm:square}. This proves the remark.
\end{proof}

\begin{remark}\label{rem:2}
Circulant Hadamard codes of length $4n=2^s$, with $s\geq 3$ do not exist.
\end{remark}
\begin{proof}
	We show that the rank of such a codes, in case they exist, should be $r>2n$ and, from \Cref{aske}, this means that $s<3$ which is out of our statement.
	
	Using the polynomial characterization we have seen in \Cref{prop:str} we have that vectors in $C$ are $\ba, \ba^2, \ldots, \ba^{4n-1}$ and the complements, where $\ba^i$, in polynomial way, is $(1+x^i)g$ and $g$ is the generator polynomial of the corresponding circulant matrix.  From \Cref{lemm:2.4} vector $\bu$ belongs to the linear span of $\langle \ba \rangle$ and so rank of code $C$ is $r= 4n-1-d$, where $d=\Degree(\Gcd(g,x^{4n}+1))$.
	Now, assume that $r\leq 2n$ and so $\Gcd(g,x^{4n}+1)=(1+x)^d$, for some integer $d=4n-1-r\geq 2n-1$ and $g=(1+x)^dq(x)$, for some polynomial $q(x)$ without common factors with $x^{4n}+1$.
	Therefore, vectors in $C$ are, in polynomial form,  $(1+x^i)(1+x)^dq(x)$, for $i\in \{0,\ldots, 4n-1\}$, and the complements. 
	Specifically, $\ba^{2n}=(1+x^{2n})(1+x)^dq(x)=(1+x)^{2n+d}q(x) \bmod{x^{4n}+1}$. 
	For $d\geq 2n$ the above equality leads to a contradiction. 
	Hence $d=2n-1$ and $r=2n$. From \Cref{aske} we obtain $s=3$ and we already know by straight checking that this case does not give any circulant Hadamard code.
	This proves the statement.
\end{proof}
\printbibliography
\end{document}